\documentclass{article}
\usepackage[utf8]{inputenc}
\usepackage{appendix}
\usepackage{float}
\usepackage{enumerate}
\usepackage{todonotes}
\usepackage{comment}

\title{$H^2$-regularity on convex domains for Robin eigenfunctions with parameter of arbitrary sign}
\author{Pier Domenico Lamberti\footnote{Università degli Studi di Padova\,, Dipartimento di Tecnica e Gestione dei Sistemi Industriali (DTG)\,, Stradella S. Nicola 3\,, 36100 Vicenza\,, Italy. Email:\, pierdomenico.lamberti@unipd.it}   and Luigi Provenzano\footnote{Sapienza Universit\`a di Roma\,, Dipartimento di Scienze di Base e Applicate per l'Ingegneria\,, Via Scarpa 16\,, 00161 Roma\,, Italy. Email:\, luigi.provenzano@uniroma1.it}}
\date{\today}

\usepackage{makeidx}
\usepackage{amssymb , amsmath, amsthm}
\usepackage{graphicx}
\usepackage{xcolor}
\newtheorem{defi}{Definition} 
\newtheorem{thm}[defi]{Theorem}

\newtheorem{rem}[defi]{Remark}
 
\newtheorem{lemme}[defi]{Lemma}
\newtheorem{cor}[defi]{Corollary}

\newcommand{\matrice}{\begin{pmatrix}}
\newcommand{\ok}{\end{pmatrix}}

\newcommand{\R}{\mathbb R}
\newcommand{\N}{\mathbb N}

\begin{document}

\maketitle

\noindent
{\bf Abstract.} We prove that the Robin eigenfunctions on convex domains of $\mathbb R^n$ are $H^2$ regular regardless of the sign of the parameter involved in the boundary conditions. The proof is an adaptation of a classical argument used in the case of positive parameters combined with a Rellich-Pohozaev identity. 
\vspace{11pt}


\vspace{6pt}

\section{Introduction}
Let $\Omega$ be a bounded domain of $\mathbb R^n$, $n\geq 2$ (i.e., a bounded connected open set) with Lipschitz boundary, and $\beta\in\mathbb R$.
In this paper we consider the following eigenvalue problem:
\begin{equation}\label{robin_class}
\begin{cases}
-\Delta u=\lambda u\,, &  {\rm in\ }\Omega\,,\\
\partial_{\nu}u+\beta u=0\,, & {\rm on\ }\partial\Omega.
\end{cases}
\end{equation}
Here $\nu$ denotes the unit outer normal to $\partial\Omega$ and $\partial_{\nu}u$ the normal derivative of $u$. Problem \eqref{robin_class} is known as the {\it Robin problem} with parameter $\beta$. For $\beta=0$ we have the Neumann eigenvalue problem, while for $\beta\to+\infty$ we recover the Dirichlet eigenvalue problem. In the literature $\beta$ is usually assumed to be positive, but in this paper we do not impose any restriction on its sign. If the domain is not regular enough, problem \eqref{robin_class} has to be understood in the weak sense as follows:
\begin{equation}\label{weak_robin}
\int_{\Omega}\nabla u\cdot\nabla\phi+\beta\int_{\partial\Omega}u\phi=\lambda\int_{\Omega}u\phi\,,\ \ \ \forall\phi\in H^1(\Omega),
\end{equation}
where the unknown $u$ belongs to the  Sobolev Space $H^1(\Omega)$.  By  $H^m(\Omega)$ we denote the usual Sobolev space of functions in $L^2(\Omega)$ with weak derivatives up to order $m$ in $L^2(\Omega)$.

\medskip

In this paper we prove the following

\begin{thm}\label{mainthm}
Let $\Omega$ be a bounded convex domain in $\mathbb R^n$ and  $u\in H^1(\Omega)$ be a solution of \eqref{weak_robin}. Then $u\in H^2(\Omega)$.
\end{thm}

Theorem \ref{mainthm} is well-known for $\beta\geq 0$ and $\beta=+\infty$ and the proof can be found, e.g., in the classical book by Grisvard \cite[Theorem 3.2.3.1]{grisvard}.  The proof in  \cite[\S 3]{grisvard} is based on the approximation of $\Omega$ by means of a sequence of smooth convex domains where the $H^2$-regularity is known by classical regularity theory, and on uniform estimates of the $H^2$-norms of the eigenfunctions on the approximating domains.  

\medskip

The condition $\beta\geq 0$ is used in a substantial way in  \cite[Theorem 3.2.3.1]{grisvard} and relaxing it is not straightforward. The same obstruction appears in the analysis of the Steklov problem and has been recently removed in \cite{lampro_steklov}. In this paper we adapt the method of \cite{lampro_steklov} to the Robin problem with $\beta<0$. Although our contribution seems to be new for $\beta<0$, in this paper we do not impose any restriction on the sign of $\beta$ for the sake of completeness. As in \cite{lampro_steklov} our proof is based on Reilly's formula combined with the Rellich-Pohozaev identity. Moreover, in order to pass to the limit in the approximation procedure, we need a spectral stability result for the Robin problem. Again,  the spectral stability is well-known for $\beta\geq 0$ (see e.g., \cite{dada97}), and here we provide a proof also for $\beta<0$ in the case under consideration, see Theorem \ref{stability_easy}.

\medskip

This paper is organized as follows.  Section~\ref{sec:pre} is devoted to preliminaries and notation. In Section~\ref{sec:estimates} we prove the $H^2$-estimates  on smooth convex domains. In Section~\ref{approximation}
we prove Theorem \ref{stability_easy} theorem concerning the spectral stability of the Robin problem and Theorem \ref{mainthm}. 

\section{Preliminaries and notation}\label{sec:pre}

It is well-known that problem \eqref{weak_robin} can be considered as an eigenvalue problem for a semi-bounded operator $T$ with compact resolvent. In fact, problem \eqref{weak_robin} is equivalent to
\begin{equation}\label{weak_robin2}
\int_{\Omega}\nabla u\cdot\nabla\phi+\beta\int_{\partial\Omega}u\phi+C\int_{\Omega}u\phi=\Lambda\int_{\Omega}u\phi\,,\ \ \ \forall\phi\in H^1(\Omega),
\end{equation}
where $C$ is any real number and $\Lambda=\lambda+C$. By Corollary \ref{coercivity} it follows that there exists a positive constant $C$ such that the quadratic form at the left-hand side of \eqref{weak_robin2} is coercive in $H^1(\Omega)$. Thus, by standard Spectral Theory there exists a self-adjoint operator $T$ such that the quadratic form associated with $T+CI$ is precisely the one at the left-hand side of \eqref{weak_robin2}. Moreover, since $\Omega$ is bounded and convex, hence Lipschitz, the embedding $H^1(\Omega)\subset L^2(\Omega)$ is compact.  Hence $T+CI$ has compact resolvent, its spectrum consists of a divergent sequence of positive eigenvalues of finite multiplicity and the corresponding eigenfunctions can be chosen to form an orthonormal basis of $L^2(\Omega)$. In particular, it follows that the eigenvalues of $T$ form a sequence of the type
\begin{equation}\label{spectrum}
-C<\lambda_1<\lambda_2\leq\cdots\leq\lambda_j\leq\cdots\nearrow+\infty.
\end{equation}
We denote by$\{u_j\}_{j=1}^{\infty}$ a corresponding orthonormal basis of $L^2(\Omega)$ of eigenfunctions. Note that if $\beta>0$ all eigenvalues are positive, while if $\beta<0$ only a finite number of eigenvalues can be negative.

It is also standard to see that the eigenvalues can be represented by means of the Min-Max Principle as follows:
\begin{equation}
\lambda_j=\min_{\substack{U\subset H^1(\Omega)\\{\rm dim}\,U=j}}\max_{0\ne u\in U}\frac{\int_{\Omega}|\nabla u|^2+\beta\int_{\partial\Omega}u^2}{\int_{\Omega}u^2}.
\end{equation}

As mentioned in the introduction, one of the main arguments in the proof of the $H^2$ regularity of the Robin eigenfunctions is the approximation of a convex domain by smooth convex domains belonging to a uniform class. Namely, it will be required that the approximating domains can be locally described near the boundaries as the subgraphs of smooth functions with uniformly bounded gradients. It turns out that those domains belong to the same Lipschitz class. We find it convenient to recall the following definition 
from \cite{BuLa} involving the notion of atlas, see
also \cite{burenkov_book}. Given
a set $V\subset \R^n$ and a number $\delta>0$ we write
\begin{equation} \label{eq:def-Vdelta-B}
V_\delta:=\{x\in V:d(x,\partial V)>\delta\} \, ,
\end{equation}
where, for $x\in\mathbb R^n$ and $A\subset\mathbb R^n$, $d(x,A):=\inf_{a\in A}|x-a|$.
\begin{defi} \label{d:atlas-B}
Let $\delta>0$, $s,s'\in \N$ with
$s'<s$. Let $\{V_j\}_{j=1}^s$ be a family of open cuboids (i.e.
rotations of rectangle parallelepipeds in $\R^n$) and
$\{r_j\}_{j=1}^s$ be a family of rotations in $\R^n$. We say that
$\mathcal A=(\delta,s,s',\{V_j\}_{j=1}^s,\{r_j\}_{j=1}^s)$ is an
atlas in $\R^n$ with parameters
$\delta,s,s',\{V_j\}_{j=1}^s,\{r_j\}_{j=1}^s$ briefly an atlas in
$\R^n$. We say that a bounded domain $\Omega\subset \R^n$ is of class
$C^{0,1}_M(\mathcal A)$ if
\begin{itemize}
\item[(i)] $\Omega\subset \cup_{j=1}^s (V_j)_\delta$ \ and \
$(V_j)_\delta \cap \Omega\neq \emptyset$ where $(V_j)_\delta$ is meant
in the sense given in \eqref{eq:def-Vdelta-B} ;

\item[(ii)] $V_j\cap \partial\Omega \neq \emptyset$ \ for \
$j=1,\dots,s'$ \ and \ $V_j\cap \partial\Omega=\emptyset$ \ for \
$s'+1\le j\le s$;

\item[(iii)] for $j=1,\dots,s$ \ we have
\begin{align*}
& r_j(V_j)=\{x\in \R^n:a_{ij}<x_i<b_{ij}\, , i=1,\dots,n\}, \quad
j=1,\dots,s ; \\
& r_j(V_j\cap \Omega)=\{x=(x',x_n)\in\R^n:x'\in W_j,
a_{nj}<x_n<g_j(x')\}, \quad j=1,\dots, s'
\end{align*}
where $x'=(x_1,\dots,x_{n-1})$, $W_j=\{x'\in
\R^{n-1}:a_{ij}<x_i<b_{ij}, i=1,\dots,n-1\}$ and the functions
$g_j\in C^{0,1}(\overline{W_j})$ for any $j\in 1,\dots,s'$ with $\|\nabla g_j\|_{\infty}\leq M$. Moreover, for $j=1,\dots,s'$
we assume that $a_{nj}+\delta\le g_j(x')\le b_{nj}-\delta$, for all
$x'\in \overline{W_j}$.
\end{itemize}

\noindent We say that a bounded domain $\Omega\subset\R^n$ is of
class $C^{0,1}(\mathcal A)$ if $\Omega$ is of class
$C^{0,1}_M(\mathcal A)$ for some $M>0$.

\noindent Finally we say that $\Omega$ if of class $C^{0,1}$
if it is of class $C^{0,1}_M(\mathcal A)$ for some atlas
$\mathcal A$ and some $M>0$.
\end{defi}

We also need the following technical lemma on a trace inequality with a small parameter.

\begin{lemme}\label{traceC}
Let $\Omega$ be a bounded domain of class $C^{0,1}_M(\mathcal A)$. For any $\varepsilon>0$ there exists a positive constant $C(\varepsilon)$ depending only on $\varepsilon,\mathcal A, M$ such that
\begin{equation}\label{traceCineq}
\int_{\partial\Omega}u^2\leq\varepsilon\int_{\Omega}|\nabla u|^2+C(\varepsilon)\int_{\Omega}u^2,
\end{equation}
for all $u\in H^1(\Omega)$.
\end{lemme}

\begin{proof}
The proof is an adaptation of the classical proof of the Trace Theorem. We sketch the main steps. Following e.g., \cite[Theorem 1, \S 5.5]{evansbook}, we begin by assuming that $u$ is a smooth function and that the boundary is flat in a neighborhood of a point $x_0\in\partial\Omega$. Namely, we assume that $\partial\Omega\cap B(x_0,r)=\{x=(x_1,...,x_n)\in B(x_0,r):x_n=0\}$. Next, we consider a non-negative smooth function $\zeta$ with compact support in $B$ such that $\zeta\equiv 1$ in $B(x_0,r/2)$. Then by the Divergence Theorem and by Cauchy's inequality we get
\begin{multline}
\int_{\partial\Omega\cap B(x_0,r/2)}u^2\leq\int_{\partial\Omega\cap B(x_0,r)}\zeta u^2=-\int_{\Omega\cap B(x_0,r)}\partial_{x_n}(\zeta u^2)\\
=-\int_{\Omega\cap B(x_0,r)}u^2\partial_{x_n}\zeta-2\int_{\Omega\cap B(x_0,r)}\zeta u\partial_{x_n}u\\
\leq \|\partial_{x_n}\zeta\|_{\infty}\int_{\Omega\cap B(x_0,r)}u^2+\varepsilon\|\zeta\|_{\infty}\int_{\Omega\cap B(x_0,r)}|\partial_{x_n}u|^2+\frac{\|\zeta\|_{\infty}}{\varepsilon}\int_{\Omega\cap B(x_0,r)}u^2\\
\leq \varepsilon\|\zeta\|_{\infty}\int_{\Omega\cap B(x_0,r)}|\partial_{x_n}u|^2+\left(\|\partial_{x_n}\zeta\|_{\infty}+\frac{\|\zeta\|_{\infty}}{\varepsilon}\right)\int_{\Omega\cap B(x_0,r)}u^2.
\end{multline}
In the general case we flatten the boundary near $x_0$ by means of a bi-Lipschitz transformation $\phi$ with $\|\nabla\phi\|_{\infty}$, $\|\nabla\phi^{-1}\|_{\infty}\leq L$ where $L$ depends only on $M$. By changing variables in integrals and rescaling $\varepsilon$, we obtain
\begin{equation}
\int_{\partial\Omega\cap B(x_0,r/2)}u^2\leq\varepsilon\int_{\Omega\cap B(x_0,r)}|\nabla u|^2+C(\varepsilon)\int_{\Omega\cap B(x_0,r)}u^2.
\end{equation}
The proof can be completed by covering the boundary of $\Omega$ with a finite number (depending only on $\mathcal A$) of balls.
\end{proof}

By Lemma \ref{traceC} we immediately deduce the following

\begin{cor}\label{coercivity}
Under the same assumptions of Lemma \ref{traceC}, we have that for  any $\beta\in\mathbb R$ there exists a non-negative constant $C$ such that the quadratic form
$$
\int_{\Omega}|\nabla u|^2+\beta\int_{\partial\Omega}u^2+C\int_{\Omega}u^2
$$
is coercive in $H^1(\Omega)$. In particular, if $\beta\geq 0$ one can choose any positive $C$, while if $\beta<0$, we have for any $\varepsilon\in(0,-1/\beta)$ that the following inequality holds:
\begin{equation}\label{coercive_N}
\int_{\Omega}|\nabla u|^2+\beta\int_{\partial\Omega}u^2-2\beta C(\varepsilon)\int_{\Omega}u^2\geq\min\{1+\beta\varepsilon,-\beta C(\varepsilon)\}\int_{\Omega}\left(|\nabla u|^2+u^2\right)
\end{equation}
for all $u\in H^1(\Omega)$.
\end{cor}
\section{$H^2$-estimates on smooth convex domains}\label{sec:estimates}

In this section we prove $H^2$ estimates of $L^2$-normalized Robin eigenfunctions on smooth convex domains $\Omega$. It turns out that these estimates depend only on $\beta$ (the Robin parameter), $\lambda$ (the eigenvalue), the dimension, the diameter, the inradius of the domain, as well as on the $L^2$ norm of the boundary traces. We recall that the diameter $D$ and the inradius $\rho$ of $\Omega$  are defined by
$$
D=\sup_{x,y\in\Omega}|x-y|\,,\ \ \ \rho=\sup_{x\in\Omega}\inf_{y\in\partial\Omega}|x-y|.
$$

In order to prove our estimates we need the Rellich-Pohozaev identity \cite{pohozaev,rellich} and the Reilly's formula \cite{Reilly}. The version of the Rellich-Pohozaev identity used in this paper is an adaptation to the Robin case of the identity proved in \cite[Lemma~3.1]{PS_steklov}.

\begin{thm}[Rellich-Poho\v{z}aev identity]\label{rellich}
Let $\Omega$ be a bounded smooth domain in $\mathbb R^n$ and let $u\in H^2(\Omega)$ be such that $-\Delta u=\lambda u$ in $L^2(\Omega)$. Then
\begin{multline}\label{rellich_id}
\frac{\lambda}{2}\int_{\partial\Omega}u^2\,x\cdot\nu-\frac{\lambda n}{2}\int_{\Omega}u^2\\+\frac{1}{2}\int_{\partial\Omega}(\partial_{\nu}u)^2\,x\cdot\nu-\frac{1}{2}\int_{\partial\Omega}|\nabla_{\partial\Omega}u|^2\,x\cdot\nu+\int_{\partial\Omega}\partial_{\nu}u\,x\cdot\nabla_{\partial\Omega}u+\frac{n-2}{2}\int_{\Omega}|\nabla u|^2=0,
\end{multline}
where $x$ denotes the position vector.
\end{thm}

\begin{proof}
From \cite[Lemma~3.1]{PS_steklov} we deduce that
\begin{equation}\label{RP1}
\int_{\Omega}\Delta u\, x\cdot\nabla u=\frac{1}{2}\int_{\partial\Omega}(\partial_{\nu}u)^2\,x\cdot\nu-\frac{1}{2}\int_{\partial\Omega}|\nabla_{\partial\Omega}u|^2\,x\cdot\nu+\int_{\partial\Omega}\partial_{\nu}u\,x\cdot\nabla_{\partial\Omega}u+\frac{n-2}{2}\int_{\Omega}|\nabla u|^2.
\end{equation}
On the other hand,
\begin{equation}
\int_{\Omega}u\,x\cdot\nabla u=\int_{\partial\Omega}u^2\, x\cdot\nu-\int_{\Omega}u\,{\rm div}(ux)=\int_{\partial\Omega}u^2\, x\cdot\nu-n\int_{\Omega}u^2-\int_{\Omega}u\, x\cdot\nabla u
\end{equation}
from which we get
\begin{equation}\label{RP2}
\int_{\Omega}u\,x\cdot\nabla u=\frac{1}{2}\int_{\partial\Omega}u^2\,x\cdot\nu-\frac{n}{2}\int_{\Omega}u^2.
\end{equation}
The proof is concluded by using the equality $\int_{\Omega}(\Delta u+\lambda u)\,x\cdot\nabla u=0$ and \eqref{RP1} and \eqref{RP2}.
\end{proof}

In what follows, by $D^2u$ we denote Hessian matrix of a function $u$  and $|D^2u|^2=\sum_{i,j=1}^n(\partial^2_{ij}u)^2$.
\begin{thm}[Reilly's formula]
Let $\Omega$ be a bounded smooth domain in $\mathbb R^n$ and let $u\in H^2(\Omega)$. Then
\begin{equation}\label{reilly_for}
\int_{\Omega}|D^2u|^2=\int_{\Omega}(\Delta u)^2-\int_{\partial\Omega}\left[(n-1)\mathcal H(\partial_{\nu}u)^2+2\Delta_{\partial\Omega}u\partial_{\nu}u+II(\nabla_{\partial\Omega}u,\nabla_{\partial\Omega}u)\right],
\end{equation}
where $\mathcal H$ is the mean curvature of the boundary, $II$ is the second fundamental form of the boundary, $\Delta_{\partial\Omega}$ and $\nabla_{\partial\Omega}$ are the boundary Laplacian and gradient, respectively.
\end{thm}


We are now ready to prove  the following
\begin{lemme}\label{hessian_smooth}
Let $\Omega$ be a bounded smooth convex domain in $\mathbb R^n$ with diameter $D$ and inradius $\rho$. Let $u$ be an eigenfunction  of \eqref{weak_robin} corresponding to the  eigenvalue $\lambda$, normalized by $\int_{\Omega}u^2=1$. Then
\begin{equation}\label{hessian_smooth_formula}
\int_{\Omega}|D^2u|^2\leq\lambda^2- 2\min\{\beta,0\}C\left(D,\rho,\beta,\int_{\partial\Omega}u^2\right)
\end{equation}
where 
\begin{equation}\label{const}
C(D,\rho,\beta,t):=\left(D|\beta|t^{1/2}+\sqrt{(D^2\beta^2-\left(\beta(n-2)-|\lambda+\beta^2|D\right)\rho )t-2\lambda\rho}\right)^2\rho^{-2}.
\end{equation}
In particular, if $\beta\geq 0$ then
\begin{equation}\label{hessian_smooth_formula3}
\int_{\Omega}|D^2u|^2\leq\lambda^2,
\end{equation}
while if $\beta<0$, then
\begin{equation}\label{hessian_smooth_formula2}
\int_{\Omega}|D^2u|^2\leq\lambda^2- 2\beta C\left(D,\rho,\beta,\frac{\varepsilon\lambda+C(\varepsilon)}{1+\varepsilon\beta}\right),
\end{equation}
for any $\varepsilon\in(0,-1/\beta)$, where $C(\varepsilon)$ is as in Lemma \ref{traceC}.
\end{lemme}
\begin{proof}
Since $\Omega$ is smooth, by classical elliptic regularity  we have that $u\in H^2(\Omega)$, see e.g. \cite[Chapter~2]{grisvard}. Following \cite{lampro_steklov} we first estimate the $L^2$-norm on $\partial\Omega$ of the tangential component of the gradient of $u$ by a constant depending only on $n,\lambda,\beta,D,\rho$ and the $L^2$-norm on $\partial\Omega$ of $u$.

Up to a translation, we may assume that $0\in\Omega$ and that $B_{\rho}$ is a ball of radius $\rho$ centered at $0$ contained in $\Omega$.
We observe that $\int_{\Omega}|\nabla u|^2=\lambda-\beta\int_{\partial\Omega}u^2$ and that $\partial_{\nu}u=-\beta u$ on $\partial\Omega$ in the sense of traces, which are well-defined in $L^2(\partial\Omega)$ since $u\in H^2(\Omega)$. After simplifications, we get from \eqref{rellich_id}
$$
2\lambda+\int_{\partial\Omega}|\nabla_{\partial\Omega}u|^2x\cdot\nu+2\beta\int_{\partial\Omega}u\,\nabla_{\partial\Omega}u\cdot x-(\lambda+\beta^2)\int_{\partial\Omega}u^2x\cdot\nu+\beta(n-2)\int_{\partial\Omega}u^2=0.
$$
Now, $|x|\leq D$ and $\rho\leq x\cdot\nu \leq D$. Note that the inequality $x\cdot\nu\geq\rho$ follows from the convexity of $\Omega$, see \cite[Lemma 5]{lampro_steklov}.
 Hence we have
$$
\frac{2\lambda}{\rho}+\int_{\partial\Omega}|\nabla_{\partial\Omega}u|^2-\frac{2|\beta| D}{\rho}\left(\int_{\partial\Omega}|\nabla_{\partial\Omega}u|^2\right)^{1/2}\left(\int_{\partial\Omega}u^2\right)^{1/2}-\frac{|\lambda+\beta^2|D}{\rho}\int_{\partial\Omega}u^2+\frac{\beta(n-2)}{\rho}\int_{\partial\Omega}u^2\leq 0
$$
which reads
$$
\frac{2\lambda}{\rho}+\int_{\partial\Omega}|\nabla_{\partial\Omega}u|^2-\frac{2|\beta| D}{\rho}\left(\int_{\partial\Omega}|\nabla_{\partial\Omega}u|^2\right)^{1/2}\left(\int_{\partial\Omega}u^2\right)^{1/2}+\frac{\beta(n-2)-|\lambda+\beta^2| D}{\rho}\int_{\partial\Omega}u^2\leq 0.
$$
This is a second order equation in $\left(\int_{\partial\Omega}|\nabla_{\partial\Omega}u|^2\right)^{1/2}$ which implies the upper bound
\begin{equation}\label{est_tang_grad}
\int_{\partial\Omega}|\nabla_{\partial\Omega} u|^2\leq C\left(D,\rho,\lambda,\int_{\partial\Omega}u^2\right).
\end{equation}
Note that $C\left(D,\rho,\lambda,\int_{\partial\Omega}u^2\right)$ is well-defined because the discriminant of the second order equation under consideration must be non-negative.

Now we use this estimate to bound the $L^2$-norm of the Hessian of $u$. To do so we first observe that the mean curvature $\mathcal H$ is non-negative on $\partial\Omega$ and that $II$ is a non-negative quadratic form on the tangent spaces to $\partial\Omega$. Then by the Reilly's formula \eqref{reilly_for} we get
\begin{multline}\label{reilly_appl}
\int_{\Omega}|D^2u|^2
=\lambda^2-\int_{\partial\Omega}[\beta^2(n-1)\mathcal H u^2-2\beta\Delta_{\partial\Omega}uu+II(\nabla_{\partial\Omega}u,\nabla_{\partial\Omega}u)]\\
\leq\lambda^2+2\beta\int_{\partial\Omega}\Delta_{\partial\Omega}uu=\lambda^2-2\beta\int_{\partial\Omega}|\nabla_{\partial\Omega}u|^2.
\end{multline}
By combining \eqref{est_tang_grad} and \eqref{reilly_appl} we deduce the validity of \eqref{hessian_smooth_formula}. Inequality \eqref{hessian_smooth_formula2} follows from  \eqref{traceCineq},  \eqref{hessian_smooth_formula} and the equality $\int_{\Omega}|\nabla u|^2=\lambda-\beta\int_{\partial\Omega}u^2$.
\end{proof}

\section{Spectral stability and proof of Theorem~\ref{mainthm}\label{approximation}}

We begin by recalling the following stability result from  \cite{ADR} and \cite[\S2]{HP}. By $d^{\mathcal{H}}(\Omega_1,\Omega_2)$ we denote the Hausdorff distance between two open sets $\Omega_1,\Omega_2$, defined by
$$
d^{\mathcal{H}}(\Omega_1,\Omega_2):=\max\{\sup_{x\in\Omega_1}d(x,\Omega_2),\sup_{y\in\Omega_2}d(y,\Omega_1)\}.
$$

\begin{lemme}\label{approximation_lem}
Let $\Omega$ be a bounded convex domain in $\mathbb R^n$ and let $\{\Omega_k\}_{k=1}^{\infty}$ be a sequence of smooth convex domains such that $\Omega\subset\Omega_k$, $\Omega_{k+1}\subset\Omega_k$ and $\lim_{k\to\infty}d^{\mathcal{H}}(\Omega,\Omega_k)=0$.   Then, the inradius and the diameter of $\Omega_k$ converge to the inradius and the diameter  of $\Omega$ respectively, as $k\to\infty$. Moreover, for any $u_k\in H^1(\mathbb R^n)$ converging weakly in $H^1(\mathbb R^n)$ to $u\in H^1(\mathbb R^n)$ we have
$$
\lim_{k\to\infty}\int_{\partial\Omega_k}u_k^2=\int_{\partial\Omega}u^2.
$$
\end{lemme}

Next we prove the following spectral stability result.

\begin{thm}\label{stability_easy}
Let $\Omega$ be a bounded convex domain in $\mathbb R^n$ and let $\Omega_k$, $k\in\mathbb{N}$, be a sequence of bounded smooth convex domains, with $\Omega\subset\Omega_{k+1}\subset\Omega_k$ for all $k\in\mathbb{N}$,  such that $\lim_{k\to+\infty}d^{\mathcal{H}}(\Omega,\Omega_k)=0$. Let $\lambda_j(k),\, \lambda_j$ denote  the Robin eigenvalues of $\Omega_k,\, \Omega$ respectively. Then
\begin{equation}\label{limh1}
\lim_{k\to \infty}\lambda_j(k)=\lambda_j,
\end{equation}
for all $j\in\mathbb N$. 
Moreover, 
if $\{u_j(k)\}_{j=1}^{\infty}$ is an orthonormal basis of $L^2(\Omega_k)$ of Robin eigenfunctions associated with $\lambda_j(k)$, then   there exists an orthonormal basis $\{u_j\}_{j=1}^\infty$ of $L^2(\Omega)$ of Robin eigenfunctions associated with $\lambda_j$ such that,
possibly passing to a subsequence with respect to $k$, we have
\begin{equation}\label{limh2}
\lim_{k\to\infty}\|u_j(k)-u_j\|_{H^1(\Omega)}=0.
\end{equation}
\end{thm}
\begin{proof}
The proof is divided in three steps.

\smallskip

{\bf Step 1.} We fix $j\in \mathbb{N}$ and we prove that $\lambda_j(k)\to\lambda_j$ as $k\to\infty$. Let $v_i$, $i=1,...,j$, be a $L^2(\Omega)$ orthonormal family of Robin eigenfunctions in $\Omega$, with associated eigenvalues $\lambda_i$. 
Let $\tilde v_i\in H^1(\mathbb R^n)$ be some extension of $v_i$ to $\mathbb R^n$, and let $\tilde V_j=\{\sum_{i=1}^ja_i\tilde v_i:\sum_{i=1}^ja_i^2=1, \ a_i\in \mathbb{R} \}$. From Lemma \ref{approximation_lem} we deduce that there exists $\varepsilon(j,k)>0$ such that
$$
1-\varepsilon_1(j,k)\leq\frac{\int_{\partial\Omega_k}\tilde v^2}{\int_{\partial\Omega}v^2}\leq1+\varepsilon_1(j,k)
$$
for all $\tilde v\in\tilde V_j$, 
where $\varepsilon_1(j,k)\to 0$ as $k\to\infty$. We have
\begin{multline}\label{EE}
\lambda_j(k)\leq\max_{\tilde v\in \tilde V_j}\frac{\int_{\Omega_k}|\nabla \tilde v|^2+\beta\int_{\partial\Omega_k}\tilde v^2}{\int_{\Omega_k}\tilde v^2}\leq\max_{\tilde v\in \tilde V_j}\frac{\int_{\Omega}|\nabla  v|^2+\beta(1+{\rm sign}(\beta)\varepsilon_1(j,k))\int_{\partial\Omega}v^2+\int_{\Omega_k\setminus\Omega}|\nabla \tilde v|^2}{\int_{\Omega}v^2}\\
\leq\max_{\tilde v\in \tilde V_j}\frac{\int_{\Omega}|\nabla  v|^2+\beta\int_{\partial\Omega} v^2}{\int_{\Omega} v^2}+\max_{\tilde v\in \tilde V_j}\frac{\int_{\Omega_k\setminus\Omega}|\nabla\tilde v|^2+{\rm sign}(\beta)\varepsilon_1(j,k)\beta\int_{\partial\Omega}v^2}{\int_{\Omega}v^2}\\
\leq \lambda_j+\max_{\tilde v\in \tilde V_j}\left(\int_{\Omega_k\setminus\Omega}|\nabla\tilde v|^2+{\rm sign}(\beta)\varepsilon_1(j,k)\beta\int_{\partial\Omega}v^2\right)\leq\lambda_j+\varepsilon_2(j,k)
\end{multline}
where $\varepsilon_2(j,k)\to 0$ as $k\to\infty$. The last inequality in \eqref{EE} follows from the absolute continuity of the Lebesgue integral combined with the fact that $\tilde V_j$ is finite dimensional.

 Now, let $\{u_j(k)\}_{j=1}^{\infty}$ be an orthonormal basis of $L^2(\Omega_k)$ of Robin eigenfunctions associated with $\lambda_j(k)$. We show that their restrictions to $\Omega$ are linearly independent. It follows by \cite[Lemma 3.2.3.2]{grisvard} that the domains $\Omega$ and $\Omega_k$ belong to the same atlas class $C^{0,1}_M(\mathcal A)$ for all $k$ sufficiently large. Hence, by inequality \eqref{EE}, by Corollary~\ref{coercivity} and by the normalization of $u_j(k)$, it follows that the norm $\|u_j(k)\|_{H^1(\Omega_k)}$ is uniformly bounded with respect to $k$.  We show that $\lim_{k\to\infty}\int_{\Omega_k\setminus\Omega}u_i(k)u_{\ell}(k)=0$ for all $i,\ell=1,...,j$. This implies that for $k$ sufficiently large, $u_j(k)$ are linearly independent in $L^2(\Omega)$.  Indeed, by the H\"{o}lder's inequality, the Sobolev Embedding Theorem and the uniform bound on the $H^1(\Omega_k)$ norms of the eigenfunctions we have 
\begin{multline}\label{estimate_strip}\left|
\int_{\Omega_k\setminus \Omega} u_i(k) u_j(k)\right|\le 
\| u_i(k)\|_{L^2(\Omega_k\setminus \Omega)}\| u_j(k)\|_{L^2(\Omega_k\setminus \Omega)}\\
\le |\Omega_k\setminus\Omega|^{1-2/p}\|  u_i(k)\|_{L^p(\Omega_k\setminus \Omega)}\|  u_j(k)\|_{L^p(\Omega_k\setminus \Omega)}\le C |\Omega_k\setminus\Omega|^{1-2/p}
\end{multline}
for some $p>2$, where $C$ is independent on $k$. From now on we assume directly that the functions $u_i(k)$ are extended to the whole $\mathbb R^n$ with norms in $H^1(\mathbb R^n)$ uniformly bounded, and that (possibly passing to a subsequence with respect to $k$) $u_i(k)$ is weakly convergent in $H^1(\mathbb R^n)$ as $k\to\infty$, and strongly in $L^2(\Omega)$. Let $V_j(k)=\{\sum_{i=1}^ja_iu_i(k):\sum_{i=1}^ja_i^2=1, \ a_i\in\mathbb R\}$. We have
\begin{multline}\label{EEE}
\lambda_j\leq\max_{v\in V_j(k)}\frac{\int_{\Omega}|\nabla v|^2+\beta\int_{\partial\Omega}v^2}{\int_{\Omega}v^2}\leq\max_{v\in V_j(k)}\frac{\int_{\Omega_k}|\nabla v|^2+\beta(1+{\rm sign}(\beta)\varepsilon_3(j,k))\int_{\partial\Omega_k}v^2}{\int_{\Omega}v^2}\\
\leq\max_{v\in V_j(k)}\frac{\int_{\Omega_k}|\nabla v|^2+\beta\int_{\partial\Omega_k}v^2}{\int_{\Omega_k}v^2}\frac{\int_{\Omega_k}v^2}{\int_{\Omega}v^2}+\varepsilon_3(j,k)|\beta|\max_{v\in V_j(k)}\frac{\int_{\partial\Omega_k}v^2}{\int_{\Omega}v^2}
\\
\leq\lambda_j(k)\max_{v\in V_j(k)}\frac{\int_{\Omega_k}v^2}{\int_{\Omega}v^2}+\varepsilon_3(j,k)|\beta|C
\leq\lambda_j(k)(1+{\rm sign}(\lambda_j(k))\varepsilon_4(j,k))+\varepsilon_3(j,k)|\beta|C
\end{multline}
with $\varepsilon_3(j,k),\varepsilon_4(k)\to 0$ as $k\to\infty$ for fixed $j$. In the second inequality we have used
$$
1-\varepsilon_3(j,k)\leq\frac{\int_{\partial\Omega}v^2}{\int_{\partial\Omega_k}v^2}\leq 1+\varepsilon_3(j,k)
$$
for all $v\in V_j(k)$. In the fourth inequality we have used, for the first summand, that 
$$
\frac{\int_{\Omega_k}|\nabla v|^2+\beta\int_{\partial\Omega_k}v^2}{\int_{\Omega_k}v^2}\leq\lambda_j(k)
$$
for all $v\in V_j(k)$; for the second summand we have used that $\int_{\Omega}u_i(k)u_h(k)\to\delta_{ih}$ as $k\to\infty$ for all $i,h=1,...,j$, and that  $\int_{\partial\Omega_k}u_i(k)^2=O(1)$ as $k\to\infty$ by \eqref{traceCineq} for all $i=1,...,j$. In the fifth inequality we have used
$$
1-\varepsilon_4(j,k)\leq\frac{\int_{\Omega_k}v^2}{\int_{\Omega}v^2}\leq 1+\varepsilon_4(j,k).
$$

Inequality \eqref{EEE} combined with \eqref{EE} implies that $\lambda_j(k)\to\lambda_j$ for all $j$ as $k\to\infty$. 

\smallskip

{\bf Step 2.} By Lemma~\ref{hessian_smooth} and the  uniform bounds for the norms in $H^1(\Omega_k)$ of the eigenfunctions (see Step 1),  we have that $\{u_j(k)\}_{k=1}^{\infty}$ is bounded in $H^2(\Omega)$. Up to extracting a subsequence, we find $u_j\in H^2(\Omega)$ such that $u_j(k)\to u_j$ in $H^1(\Omega)$. We now show that $u_j$ is an eigenfunction with eigenvalue $\lambda_j$. Let $\Phi$ be a Lipschitz continuous function defined in $\mathbb R^n$. Then
\begin{equation}\label{weak_lim}
\int_{\Omega_k}\nabla u_j(k)\cdot\nabla\Phi+\beta\int_{\partial\Omega_k}u_j(k)\Phi=\lambda_j(k)\int_{\Omega_k}u_j(k)\Phi.
\end{equation}
We consider the boundary integral in  the left-hand side of \eqref{weak_lim}, and  we write
$$
\int_{\partial\Omega_{k}}u_j(k)\Phi=\int_{\partial\Omega}u_j\Phi+\left(\int_{\partial\Omega_{k}}u_j(k)\Phi-\int_{\partial\Omega}u_j(k)\Phi\right)+\left(\int_{\partial\Omega}u_j(k)\Phi-\int_{\partial\Omega}u_j\Phi\right).
$$
The second term in the right-hand side goes to zero as $k\to\infty$ thanks to Lemma \ref{approximation_lem}, while the third term goes to zero (possibly passing to a subsequence) from the compactness of the trace operator. For the volume integral in left-hand side of \eqref{weak_lim}, we have
\begin{equation}\label{last}
\int_{\Omega_k}\nabla u_j(k)\cdot\nabla\Phi=\int_{\Omega}\nabla u_j\cdot\nabla\Phi+\int_{\Omega_k\setminus\Omega}\nabla u_j(k)\cdot\nabla\Phi+\left(\int_{\Omega}\nabla (u_j(k)-u_j)\cdot\nabla\Phi\right).
\end{equation}
The second term in the right-hand side of \eqref{last} goes to zero as $k\to\infty$ because 
$$
\int_{\Omega_k\setminus\Omega}\nabla u_j(k)\cdot\nabla\Phi\leq\|\nabla u_j(k)\|_{L^2(\Omega_k\setminus\Omega)}\|\nabla \Phi\|_{L^2(\Omega_k\setminus\Omega)}\leq C\|\nabla \Phi\|_{L^2(\Omega_k\setminus\Omega)}
$$
and because $\|\nabla \Phi\|_{L^2(\Omega_k\setminus\Omega)}$ goes to zero as before by the absolute continuity of the Lebesgue integral. The third term in \eqref{last} goes to zero since $u_j(k)\to u_j$ in $H^1(\Omega)$. In the same way one can see that the volume integral in the right-hand side of \eqref{weak_lim} converges to $\int_{\Omega}u_j\Phi$ as $k\to\infty$. In conclusion, 
\begin{equation}\label{roma}
    \int_{\Omega}\nabla u_j\cdot \nabla \Phi+\beta\int_{\partial\Omega}u_j\Phi =\lambda_j \int_{\Omega}u_j\Phi,
\end{equation}
and 
 \eqref{limh2} holds. Since \eqref{roma} holds for all Lipschitz function $\Phi$ defined in $\mathbb R^n$, by a density argument we conclude that \eqref{roma} holds for all $\Phi\in H^1(\Omega)$, hence $u_j$ is a Robin eigenfunction in $\Omega$.

\smallskip

{\bf Step 3.} 
We prove that $\{u_j\}_{j=1}^{\infty}$ is an orthonormal basis of $L^2(\Omega)$. This follows simply by passing to the limit in the equality 
$
\int_{\Omega_k} u_i(k)u_j(k)=\delta_{ij}
$
as $k\to \infty$
in order to get 
$
\int_{\Omega} u_i u_j=\delta_{ij}.
$
In fact, writing 
$$\int_{\Omega_k}u_i(k) u_j(k)=\int_{ \Omega}u_i(k) u_j(k)+\int_{\Omega_k\setminus \Omega}u_i(k)u_j(k)
$$
we have that the first integral in  the right-hand side of the previous equality converges to 
$
\int_{\Omega} u_i u_j
$
while the second integral converges to zero by \eqref{estimate_strip}.

\end{proof}

We are ready to prove Theorem \ref{mainthm}. 

\begin{proof}[Proof of Theorem \ref{mainthm}]
By \cite[Lemma 3.2.1.1]{grisvard} there exists a sequence of bounded smooth convex domains $\{\Omega_k\}_{k=1}^{\infty}$ such that $\Omega\subset\Omega_{k+1}\subset\Omega_k$ for all $k\in\mathbb{N}$, and such that $\lim_{k\to\infty}d^{\mathcal{H}}(\Omega_k,\Omega)=0$. Let $\lambda_j(k)$ denote  the Robin eigenvalues of $\Omega_k$ and let $\{u_j(k)\}_{j=1}^{\infty}$ be an orthonormal basis of $L^2(\Omega_k)$ of corresponding eigenfunctions. By Theorem \ref{stability_easy},  $\lambda_j(k)\to\lambda_j$ as $k\to\infty$, where $\lambda_j$ are the Robin eigenvalues on $\Omega$, and there exists an orthonormal basis $\{u_j\}_{j=1}^{\infty}$ of $L^2(\Omega)$ of  eigenfunctions associated with the $\lambda_j$'s such that \eqref{limh2} holds, possibly passing to a subsequence with respect to $k$.

 Consider any domain $\omega$ with $\overline\omega \subset \Omega$. Now, we have $-\Delta(u_j(k)-u_j)=\lambda_j(k)u_j(k)-\lambda_ju_j$ and  by \eqref{limh2} it follows that $\|\lambda_j(k)u_j(k)-\lambda_ju_j\|_{H^1(\Omega)}\to 0$ as $k\to\infty$. It follows by elliptic regularity (see e.g., \cite[Theorem 8.10]{GT}) that $\lim_{k\to\infty}\|u_j(k)-u_j\|_{H^2(\omega)}=0$. In particular, by \eqref{hessian_smooth_formula3} if $\beta\geq 0$ we have
\begin{equation}\label{est_positive}
\int_{\omega}|D^2u_j|^2=\lim_{k\to\infty}\int_{\omega}|D^2u_j(k)|^2\leq \lim_{k\to\infty}\lambda_j(k)^2=\lambda_j^2.
\end{equation}
If $\beta<0$, by \eqref{hessian_smooth_formula2} we have that
\begin{multline}\label{est_negative}
\int_{\omega}|D^2u_j|^2=\lim_{k\to\infty}\int_{\omega}|D^2u_j(k)|^2\leq \lim_{k\to\infty}\left(\lambda_j(k)^2- 2\beta C\left(D(k),\rho(k),\beta,\frac{\varepsilon\lambda_j(k)+C(\varepsilon)}{1+\varepsilon\beta}\right)\right)\\
=\lambda_j^2- 2\beta C\left(D,\rho,\beta,\frac{\varepsilon\lambda_j+C(\varepsilon)}{1+\varepsilon\beta}\right),
\end{multline}
where $\varepsilon$ is any fixed constant satisfying $\varepsilon\in(0,-1/\beta)$. Here $D(k),\rho(k)$ are the diameter and inradius of $\Omega_k$, while $D,\rho$ are the diameter and inradius of $\Omega$.

 Consider now a sequence of domains $\omega_{\ell}\subset\Omega$ such that $\overline\omega_{\ell}\subset\Omega$, $\omega_{\ell}\subset \omega_{\ell+1}$, $\cup_{\ell=1}^{\infty}\omega_{\ell}=\Omega$.  Then, by considering \eqref{est_positive} and \eqref{est_negative} with $\omega$ replaced by $\omega_{\ell}$ and passing to the limit as $\ell\to\infty$, we get that
 \begin{align}
 & \int_{\Omega}|D^2u_j|^2\leq \lambda_j,  & {\rm if\ }\beta\geq 0,\label{convexB+}\\
& \int_{\Omega}|D^2u_j|^2\leq\lambda_j^2- 2\beta C\left(D,\rho,\beta,\frac{\varepsilon\lambda_j+C(\varepsilon)}{1+\varepsilon\beta}\right), & {\rm if\ }\beta< 0,
\end{align}
hence $u_j\in H^2(\Omega)$. Since the traces of $\{u_j\}_{j=1}^{\infty}$ form  a complete system  in $L^2(\Omega)$, all  other eigenfunctions are  linear combinations of a finite number of those eigenfunctions  hence they belong to $H^2(\Omega)$.
\end{proof}

\begin{rem}
We remark that formula \eqref{convexB+} is the  estimate found in \cite[\S3]{grisvard}.
\end{rem}





\section*{Acknowledgments}
The authors are very thankful to Professor David Krej\v ci\v r\'ik for encouraging the use of the method developed in \cite{lampro_steklov} for the case of the Robin problem with negative parameter. The second author is grateful to the Dipartimento di Scienze di Base e Applicate per l'Ingegneria of Sapienza University of Rome for the kind hospitality during the preparation of the manuscript. The first author is a member of the Gruppo Nazionale per l'Analisi  Matematica, la Probabilit\`{a} e le loro Applicazioni (GNAMPA) of the Istituto Nazionale di Alta Matematica (INdAM). The second author acknowledges the support of the INdAM GNSAGA. The authors aknowledge financial support from the project ``Perturbation problems and asymptotics for elliptic differential equations: variational and potential theoretic methods'' funded by the European Union – Next Generation EU and by MUR-PRIN-2022SENJZ3. 

\medskip

The authors declare no conflicts of interest. No data are associated with this article. The authors have no relevant financial or non-financial interests to disclose. 

\bibliography{bibliography}{}
\bibliographystyle{abbrv}
\end{document}